\documentclass[hidelinks,onefignum,onetabnum]{siamart251216}

\usepackage{amsmath,amssymb,amsfonts,booktabs,algorithm,algorithmic}
\usepackage{array}
\newsiamremark{remark}{Remark}
\newsiamremark{example}{Example}

\newcommand{\mcM}{\mathcal{M}}
\newcommand{\Vd}{V^{(d)}}
\newcommand{\binomc}[2]{\binom{#1}{#2}}

\headers{Multi-index of multivariate Vandermonde matrices}
        {Samuel Santos-P\'erez}
\title{Multi-Index Recovery for Total-Degree Multivariate Vandermonde Matrices%
\thanks{Submitted to the editors DATE.}}
\author{Samuel Santos-P\'erez\thanks{Department of Mathematics, University of Valencia,
Valencia, Spain (\email{samuel.santos@uv.es}).}}

\begin{document}
\maketitle

\begin{abstract}
For multivariate polynomial bases of total degree at most n, a column of the Vandermonde matrix is identified by a multi-index of exponents. In this work, we give a direct, recursive map from the scalar column index to this multi-index, based on a prescribed ordering in which exponents of the last variable are increased at the block level and the same rule is applied recursively to the remaining variables. The map is first derived explicitly in the bivariate case, where it yields closed-form column indices for the univariate Vandermonde matrices, and is then extended recursively to arbitrary dimension by means of a binomial counting function, with a proof of correctness by induction. The procedure provides random access to any total-degree column without forming the full tensor-product matrix, and is particularly advantageous for streaming construction and matrix-free evaluation. It is basis independent and applies directly to Chebyshev--Vandermonde matrices. For the full dense matrix, nested loops over the exponent multi-indices achieve the same asymptotic cost; the gain of the present approach lies in exact indexing and the avoidance of superfluous intermediate arrays.
\end{abstract}



\section{Introduction}

Vandermonde matrices are a fundamental tool in polynomial interpolation,
least-squares fitting, and approximation~\cite{GolubVanLoan2013}. In the multivariate case, a
practical difficulty is not merely evaluating a monomial, but establishing a
consistent mapping from a scalar column index to a multi-index of exponents
within the chosen multivariate polynomial space. For the space of polynomials
of total degree at most \(n\) in \(d\) variables, there are \(\binom{n+d}{d}\)
basis functions. In contrast, the full row-wise Kronecker product (Khatri--Rao
product) of \(d\) univariate Vandermonde matrices of degree \(n\) yields
\((n+1)^d\) columns, corresponding to monomials with individual degree \(\le n\)
in each variable, reaching a total degree of up to \(dn\). Consequently, the
full product contains numerous columns that do not belong to the total-degree
\(\le n\) space and must be discarded. This work introduces a multi-index
recovery procedure that directly yields the univariate column indices
associated with any column of the multivariate matrix, avoiding the
construction of the full Kronecker product and the subsequent selection of
superfluous columns.

Several complementary numerical strategies are used in this setting. By using
Newton forms, one can avoid a direct monomial representation in multivariate
interpolation~\cite{HechtEtAl2019}; orthogonal or Arnoldi-generated bases are
used to address conditioning issues~\cite{BrubeckNakatsukasaTrefethen2021,
KuianReichelShiyanovskii2020}; and Chebyshev nodes and bases have long been
studied for their favorable Vandermonde structure~\cite{EisinbergFranzePugliese1998,
Li2008}. The row-wise Kronecker product is itself the standard tool for
constructing multivariate Vandermonde matrices in polynomial interpolation and
regression~\cite{Safak2012,Drmac2020}, and nested loops over the exponent
multi-indices are a standard approach for generating monomials in total-degree
order~\cite{deBoor2000}. Here we give a precise, recursive indexing rule for the
total-degree ordering. It can be used with the monomial basis, or applied
unchanged when a more stable polynomial basis such as Chebyshev is chosen; it
is also applicable to both interpolation and least-squares regression problems,
and can be combined with QR or SVD factorizations.

Our first contribution is a formula for the bivariate case that relates the
columns of the multivariate Vandermonde matrix, with the specific column order
defined in Section~\ref{sec:prelim}, to those of the univariate Vandermonde
matrices. Specifically, for a given row \(i\) (a data point) and global column
\(j\) (corresponding to a monomial in the total-degree order), the formula
yields the associated univariate column indices \(m(j)\) and \(\ell(j)\) for
\(x\) and \(y\), respectively. Our second contribution is a
dimension-independent recursive extension of this mapping, together with an
induction proof that clarifies the role of the degree bound. We also analyze
the computational implications: while the full-matrix construction is
dominated by the cost of writing the entries, our procedure is particularly
advantageous for random column access, streaming construction, and matrix-free
evaluation, where the monomial associated with a given column must be
identified without enumerating the preceding columns.

\section{Polynomial space, order, and Vandermonde matrices}\label{sec:prelim}
Let \(\mathbb{N}_0 := \mathbb{N} \cup \{0\}\), let \(n \in \mathbb{N}_0\), and let \(d \geq 1\). For a multi-index \(\boldsymbol{a} = (a_1,\ldots,a_d) \in \mathbb{N}_0^d\), write
\[
\boldsymbol{x}^{\boldsymbol{a}} = x_1^{a_1} \cdots x_d^{a_d}, \qquad |\boldsymbol{a}| := a_1 + \cdots + a_d.
\]
Consider the polynomial space
\[
\mathcal{M}_{n,d} = \{\boldsymbol{x}^{\boldsymbol{a}} : \boldsymbol{a} \in \mathbb{N}_0^d,\ |\boldsymbol{a}| \leq n\}.
\]
The number of monomials of total degree exactly \(k\) in \(d\) variables equals the number of non-negative integer solutions of
\[
a_1 + \cdots + a_d = k.
\]
By the stars-and-bars argument, this number is
\[
\binom{k+d-1}{d-1}.
\]
Summing over \(k=0,1,\ldots,n\) and using the
hockey-stick identity in the form
\[
\sum_{i=r}^{m}\binomc{i}{r}=\binomc{m+1}{r+1},
\]
with \(i=k+d-1\), \(r=d-1\), and \(m=n+d-1\), we obtain
\[
|\mathcal{M}_{n,d}|
=\sum_{k=0}^{n}\binomc{k+d-1}{d-1}
=\sum_{i=d-1}^{n+d-1}\binomc{i}{d-1}
=\binomc{n+d}{d}.
\]

\paragraph{Ordering of monomials.}
For two exponent tuples
\(\boldsymbol{a} = (a_1,...,a_d)\) and \(\boldsymbol{b} = (b_1,...,b_d)\)
in \(\mathbb{N}_0^d\), the \emph{lexicographic order} is defined by comparing
the tuples componentwise from left to right: we write \(\boldsymbol{a} <_{\mathrm{lex}} \boldsymbol{b}\)
if there exists an index \(k \in \{1,\ldots,d\}\) such that \(a_i = b_i\) for all
\(i < k\) and \(a_k < b_k\).

For monomials, we consider two variants.

\begin{itemize}
    \item \emph{Direct lexicographic order}: compare the exponent tuples
    \((a_1,...,a_d)\) using \(<_{\mathrm{lex}}\). The exponent of the first
    variable is the primary key, then the exponent of the second, and so on.
    \item \emph{Inverse lexicographic order}: compare the reversed exponent
    tuples \((a_d,...,a_1)\) using \(<_{\mathrm{lex}}\). The exponent of
    the last variable is the primary key, then the exponent of the previous
    variable, and so on. Consequently, all monomials with \(a_d=0\) appear
    before those with \(a_d=1\), regardless of the other exponents.
\end{itemize}

We adopt the \emph{inverse lexicographic order} throughout, restricted to
monomials of total degree at most \(n\).

\begin{example}[\(d=2\), \(n=2\)]\label{ex:order-2d}
The monomials \(x^a y^b\) with \(a+b \le 2\) ordered by the direct
lexicographic order on \((a,b)\) are
\[
1,\ y,\ y^2,\ x,\ xy,\ x^2.
\]
Ordered by the inverse lexicographic order, equivalently by the direct
lexicographic order on the reversed tuple \((b,a)\), they are
\[
1,\ x,\ x^2,\ y,\ xy,\ y^2.
\]
\end{example}
\begin{example}[\(d=3\), \(n=2\)]\label{ex:order-3d}
The monomials \(x^a y^b z^c\) with \(a+b+c \le 2\) ordered by the direct
lexicographic order on \((a,b,c)\) are
\[
1,\ z,\ z^2,\ y,\ yz,\ y^2,\ x,\ xz,\ xy,\ x^2.
\]
Ordered by the inverse lexicographic order they are
\[
1,\ x,\ x^2,\ y,\ xy,\ y^2,\ z,\ xz,\ yz,\ z^2.
\]
\end{example}
\begin{example}[\(d=3\), \(n=3\)]
The monomials \(x^a y^b z^c\) with \(a+b+c \le 3\) ordered by the direct
lexicographic are
\[
1,\ z,\ z^2,\ z^3,\ y,\ yz,\ yz^2,\ y^2,\ y^2z,\ y^3,\
x,\ xz,\ xz^2,\ xy,\ xyz,\ xy^2,\ x^2,\ x^2z,\ x^2y,\ x^3.
\]
Ordered by the inverse lexicographic order they are
\[
1,\ x,\ x^2,\ x^3,\ y,\ xy,\ x^2y,\ y^2,\ xy^2,\ y^3,\ z,\ xz,\ x^2z,\ yz,\ xyz,\ y^2z,\ z^2,\ xz^2,\ yz^2,\ z^3.
\]
\end{example}
\begin{example}[\(d=4\), \(n=2\)]
The monomials \(x^a y^b z^c w^d\) with \(a+b+c+d \le 3\) ordered by the inverse lexicographic order are
\[
1,\ x,\ x^2,\ y,\ xy,\ y^2,\ z,\ xz,\ yz,\ z^2,w,\ xw,\ yw,\ zw,\ w^2.
\]
\end{example}
In particular, suppose the exponents of the later variables
\(x_{i+1},\ldots,x_d\) have already been fixed, and let
\[
N = n - (a_{i+1}+\cdots+a_d)
\]
be the remaining degree available for the first \(i\) variables. The monomials
in these first \(i\) variables of total degree at most \(N\) are ordered by
the same rule as before, with \(x_i\) playing the role of the last variable.
Thus the blocks are ordered primarily by \(a_i\): all columns with \(a_i=0\)
precede those with \(a_i=1\), those with \(a_i=1\) precede those with \(a_i=2\),
and so on up to \(a_i=N\). Within each value of \(a_i\), the exponents of
\(x_1,\ldots,x_{i-1}\) vary according to inverse lexicographic ordering.
Consequently, the maximum permissible value \(a_i=N\) appears last
in this block. Here, a \emph{level-\(i\) block with exponents \((a_{i+1},\ldots,a_d)\)}
is the set of columns of the Vandermonde matrix that share those exponents for
the later variables \(x_{i+1},\ldots,x_d\). Within such a block, the exponents
of \(x_1,\ldots,x_i\) vary according to the ordering. This convention,
together with the fact that all matrix indices start at one, is essential for
the strict inequalities used later.

\begin{example}[Level-\(2\) block for \(d=3\), \(n=3\)]
\label{ex:level-block}
Consider the case \(d=3\), \(n=3\), with variables \(x,y,z\) and exponents
\((a,b,c)\). Fix the exponent of the last variable \(z\) to be \(c=1\). The
corresponding level-\(2\) block with exponent \(c=1\) is the set of columns
whose monomials have \(z^1\) as a factor. These are
\[
z,\ xz,\ x^2z,\ yz,\ xyz,\ y^2z.
\]
They correspond to the monomials in \(x,y\) of total degree at most
\(N=n-c=2\), ordered by the inverse lexicographic order on \((a,b)\), and
then multiplied by \(z\). Within this block, the exponent \(b\) of \(y\)
acts as the primary key: the sub-block with \(b=0\) consists of
\(z, xz, x^2z\); the sub-block with \(b=1\) consists of \(yz, xyz\); and the
sub-block with \(b=2\) consists of \(y^2z\).
\end{example}

For data points \(\boldsymbol{x}_i=(x_{i1},\ldots,x_{id})\), \(1\leq i\leq r\),
let \(\Vd\in\mathbb{R}^{r\times\binom{n+d}{d}}\) be the multivariate Vandermonde matrix in the order defined above. Let \(j\in\{1,\ldots,\binom{n+d}{d}\}\) be a global column index, and let
\[
\boldsymbol{a}(j) = (a_1(j),\ldots,a_d(j)) \in \mathbb{N}_0^d
\]
be the multi-index of exponents associated with column \(j\), so that column \(j\) represents the monomial \(\boldsymbol{x}^{\boldsymbol{a}(j)} = x_1^{a_1(j)}\cdots x_d^{a_d(j)}\). Then
\begin{equation}\label{eq:vd-entry}
 (\Vd)_{ij}=\prod_{k=1}^d x_{ik}^{a_k(j)}.
\end{equation}
The univariate Vandermonde matrices \(V_x\in\mathbb{R}^{r\times(n+1)}\) are defined by
\begin{equation}\label{eq:univ}
 (V_x)_{ij}=x_{i}^{j-1},\qquad 1\leq j\leq n+1.
\end{equation}

\section{The block-counting function}

For \(N\geq0\), \(s\geq1\), and \(0\leq q\leq N+1\), define
\begin{equation}\label{eq:f}
 f(N,s,q):=\binomc{N+s}{s}-\binomc{N+s-q}{s}.
\end{equation}
In the ordering introduced above, monomials are grouped by the exponent of
the last variable. Therefore, \(f(N,s,q)\) counts the monomials in \(s\)
variables of total degree at most \(N\) whose exponent in the last variable is
strictly smaller than \(q\), i.e., the number of columns that precede the
block with last exponent \(q\).

To derive \eqref{eq:f}, group the monomials by the value \(t\) of the exponent
of the last variable, for \(t=0,1,\ldots,q-1\). For a fixed \(t\), the
remaining \(s-1\) variables must satisfy a total degree bound of \(N-t\), so
the number of such monomials is
\[
 \binomc{N-t+s-1}{s-1}.
\]
Summing over \(t\) gives
\[
 \sum_{t=0}^{q-1}\binomc{N-t+s-1}{s-1}
 =\binomc{N+s}{s}-\binomc{N+s-q}{s}.
\]

\begin{proposition}\label{prop:block}
For \(1\leq j\leq\binomc{N+s}{s}\), there is a unique
\(q\in\{0,\ldots,N\}\) satisfying
\begin{equation}\label{eq:block}
 f(N,s,q)<j\leq f(N,s,q+1).
\end{equation}
\end{proposition}
\begin{proof}
For \(q=0,\ldots,N\), the successive difference is
\[
 f(N,s,q+1)-f(N,s,q)=\binomc{N-q+s-1}{s-1}>0.
\]
Thus \(f(N,s,\cdot)\) is strictly increasing. Since
\(f(N,s,0)=0\) and \(f(N,s,N+1)=\binomc{N+s}{s}\), the intervals
\((f(N,s,q),f(N,s,q+1)]\) for \(q=0,\ldots,N\) are disjoint and their union
is \((0,\binomc{N+s}{s}]\). Hence every integer
\(j\in\{1,\ldots,\binomc{N+s}{s}\}\) lies in exactly one of these intervals,
which proves existence and uniqueness of \(q\).
\end{proof}

\begin{corollary}\label{cor:local}
Let \(j\) be an integer with \(1\leq j\leq\binomc{N+s}{s}\), and let \(q\) be
the unique integer from Proposition~\ref{prop:block}. Define
\[
 j'=j-f(N,s,q).
\]
Then
\[
 1\leq j'\leq\binomc{N-q+s-1}{s-1}.
\]
\end{corollary}
\begin{proof}
From \eqref{eq:block}, we have
\(0<j-f(N,s,q)\leq f(N,s,q+1)-f(N,s,q)\). Using the difference formula from
the proof of Proposition~\ref{prop:block}, this becomes
\[
 0<j'\leq\binomc{N-q+s-1}{s-1}.
\]
Since \(j'\) is an integer, the lower bound is \(1\).
\end{proof}

\begin{remark}\label{rem:local-index}
The integer \(j'\) is the position of \(j\) within the block of monomials
whose exponent in the last variable equals \(q\). After fixing this exponent,
the remaining monomials in \(s-1\) variables of total degree at most \(N-q\)
are listed in the same order. There are \(\binomc{N-q+s-1}{s-1}\) such
monomials. The first one corresponds to \(j'=1\), the second to \(j'=2\), and
so on. This is the index used in the recursive step of the
algorithm.
\end{remark}

\section{The bivariate formula}

For \(d=2\), the order is
\[
 1,x,\ldots,x^n;\quad y,xy,\ldots,x^{n-1}y;\quad\ldots;\quad y^n.
\]
Thus \(f(n,2,q)=q(n+1)-q(q-1)/2\). We now state the bivariate indexing
theorem in the notation introduced above. Let \(V^{(2)}\) be the multivariate
Vandermonde matrix whose columns are ordered as above, so that its \(j\)-th
column corresponds to the \(j\)-th monomial in the order established throughout the manuscript.

\begin{theorem}\label{thm:two}
Let \(1\leq j\leq\binomc{n+2}{2}\), and let \(q(j)\) be the unique integer in
\(\{0,\ldots,n\}\) such that
\begin{equation}\label{eq:qdef}
 f(n,2,q(j))<j\leq f(n,2,q(j)+1).
\end{equation}
Define
\begin{equation}\label{eq:ml}
 \ell(j)=q(j)+1,\qquad m(j)=j-f(n,2,q(j)).
\end{equation}
Then \(1\leq m(j),\ell(j)\leq n+1\), the \(j\)-th monomial in the
total-degree order defined above is \(x^{m(j)-1}y^{\ell(j)-1}\), and
\begin{equation}\label{eq:bivfact}
 (V^{(2)})_{i,j}=(V_x)_{i,m(j)}(V_y)_{i,\ell(j)}.
\end{equation}
\end{theorem}
\begin{proof}
We prove the claims one by one.

First, recall from Proposition~\ref{prop:block} that for \(N=n\) and \(s=2\)
the function \(f(n,2,\cdot)\) is strictly increasing and satisfies
\[
f(n,2,0)=0,\qquad
f(n,2,n+1)=\binomc{n+2}{2}.
\]
Hence there is a unique integer \(q(j)\in\{0,\ldots,n\}\) such that
\[
f(n,2,q(j))<j\leq f(n,2,q(j)+1).
\]

By definition, \(f(n,2,q)\) counts the monomials \(x^a y^b\) with
\(a+b\leq n\) and \(b<q\). Therefore the columns whose monomials have
\(y\)-exponent exactly \(q\) occupy the global indices \(j\):
\[
f(n,2,q)+1,\quad f(n,2,q)+2,\quad\ldots,\quad f(n,2,q+1).
\]
Since \(j\) lies in this interval, the \(j\)-th monomial has \(y\)-exponent
equal to \(q(j)\).

Next, the local position of \(j\) inside this block is
\[
j-f(n,2,q(j)).
\]
Because indices start at \(1\), the first column of the block corresponds to
\(x\)-exponent \(0\), the second to \(x\)-exponent \(1\), and so on. Thus
\[
m(j):=j-f(n,2,q(j))
\]
is precisely the column index in \(V_x\) associated with the \(x\)-part of the
monomial. Consequently, the \(x\)-exponent is \(m(j)-1\).

For the \(y\)-part, the exponent is \(q(j)\), so the corresponding column
index in \(V_y\) is one plus that exponent:
\[
\ell(j):=q(j)+1.
\]
Hence the \(y\)-exponent is \(\ell(j)-1\).

The bounds follow immediately: since \(q(j)\in\{0,\ldots,n\}\), we have
\[
1\leq \ell(j)=q(j)+1\leq n+1.
\]
Moreover, because \(j\) belongs to the block of \(y\)-exponent \(q(j)\),
\[
1\leq j-f(n,2,q(j))\leq f(n,2,q(j)+1)-f(n,2,q(j))=n-q(j)+1\leq n+1,
\]
so \(1\leq m(j)\leq n+1\). Therefore the \(j\)-th monomial is
\[
x^{m(j)-1}y^{\ell(j)-1}.
\]

Finally, for a data point \((x_i,y_i)\), the value of this monomial is
\[
x_i^{m(j)-1}y_i^{\ell(j)-1}.
\]
By the definition of the univariate Vandermonde matrices,
\[
(V_x)_{i,m(j)}=x_i^{m(j)-1},\qquad
(V_y)_{i,\ell(j)}=y_i^{\ell(j)-1}.
\]
Multiplying these two entries gives
\[
(V^{(2)})_{i,j}
=
x_i^{m(j)-1}y_i^{\ell(j)-1}
=
(V_x)_{i,m(j)}\,(V_y)_{i,\ell(j)}.
\]
This proves \eqref{eq:bivfact}.
\end{proof}

\begin{example}\label{ex:two}
For \(n=3\), the ten columns and their univariate indices are displayed below.
\begin{center}
\small
\setlength{\tabcolsep}{2pt}
\newcolumntype{C}{>{\centering\arraybackslash}p{0.6cm}}
\begin{tabular}{c *{10}{C}}
\toprule
\(j\)&1&2&3&4&5&6&7&8&9&10\\
\midrule
monomial&\(1\)&\(x\)&\(x^2\)&\(x^3\)&\(y\)&\(xy\)&\(x^2y\)&\(y^2\)&\(xy^2\)&\(y^3\)\\
\(m(j)\)&1&2&3&4&1&2&3&1&2&1\\
\(\ell(j)\)&1&1&1&1&2&2&2&3&3&4\\
\bottomrule
\end{tabular}
\end{center}
\end{example}

\section{Arbitrary dimension}

Define \(\Phi_s(N,j)=(a_1,\ldots,a_s)\) recursively.  For \(s=1\), set
\begin{equation}\label{eq:phi1}
 \Phi_1(N,j)=(j-1),\qquad 1\leq j\leq N+1.
\end{equation}
For \(s>1\), let \(a_s\) be the unique integer satisfying
\begin{equation}\label{eq:as}
 f(N,s,a_s)<j\leq f(N,s,a_s+1),
\end{equation}
put \(N'=N-a_s\) and \(j'=j-f(N,s,a_s)\), and define
\begin{equation}\label{eq:phi-rec}
 \Phi_s(N,j)=\bigl(\Phi_{s-1}(N',j'),a_s\bigr).
\end{equation}

\begin{theorem}\label{thm:general}
For every \(s\geq1\), \(N\geq0\), and
\(1\leq j\leq\binomc{N+s}{s}\), \(\Phi_s(N,j)\) is the exponent vector of the
\(j\)-th monomial of \(\mcM_{N,s}\) in the prescribed order.  If
\(\Phi_d(n,j)=(a_1(j),\ldots,a_d(j))\) and
\(m_k(j)=a_k(j)+1\), then
\begin{equation}\label{eq:general-factor}
 (\Vd)_{i,j}=\prod_{k=1}^d (V_{x_k})_{i,m_k(j)}.
\end{equation}
\end{theorem}
\begin{proof}
We prove the statement by induction on \(s\), universally in the degree bound
\(N\). That is, for each \(s\geq 1\) we prove that \(\Phi_s(N,j)\) is correct
for every \(N\geq 0\) and every valid index \(j\).

The cases \(s=1\) and \(s=2\) are \eqref{eq:phi1} and
Theorem~\ref{thm:two}, respectively. We therefore assume \(s\geq 3\) and that
the statement holds for \(s-1\) variables for every degree bound \(N\geq 0\).

Let \(N\geq 0\) and \(1\leq j\leq\binomc{N+s}{s}\). By
Proposition~\ref{prop:block}, there exist a unique
\(a_s\in\{0,\ldots,N\}\), a reduced degree \(N'=N-a_s\), and a local index
\(j'=j-f(N,s,a_s)\) such that
\[
1\leq j'\leq\binomc{N'+s-1}{s-1}.
\]
Thus \(j'\) is a valid index for the problem of \(s-1\) variables with total
degree at most \(N'\). By the induction hypothesis,
\[
\Phi_{s-1}(N',j')=(a_1,\ldots,a_{s-1})
\]
gives the correctly ordered exponent vector for that reduced problem.

It remains to show that \(\Phi_s(N,j)=(a_1,\ldots,a_{s-1},a_s)\) is the exponent
vector of the \(j\)-th monomial in \(\mathcal{M}_{N,s}\). By
Proposition~\ref{prop:block}, the columns whose last exponent equals \(a_s\)
occupy the global indices
\[
f(N,s,a_s)+1,\quad f(N,s,a_s)+2,\quad\ldots,\quad f(N,s,a_s+1).
\]
Since \(j\) lies in this interval and \(j'=j-f(N,s,a_s)\), the integer \(j'\)
is exactly the local position of \(j\) inside the \(a_s\)-block. The ordering
inside this block is, by definition, the same as the ordering of monomials in
\(s-1\) variables with total degree at most \(N'=N-a_s\). Therefore, by the
induction hypothesis, the \(j'\)-th monomial of that reduced problem is
\[
x_1^{a_1}\cdots x_{s-1}^{a_{s-1}}.
\]
Appending \(a_s\) gives the monomial
\[
x_1^{a_1}\cdots x_{s-1}^{a_{s-1}}x_s^{a_s}.
\]
Its local position within the \(a_s\)-block is \(j'\), and the block itself
starts at global offset \(f(N,s,a_s)\). Hence its global position is
\[
f(N,s,a_s)+j'=f(N,s,a_s)+\bigl(j-f(N,s,a_s)\bigr)=j.
\]
Thus \(\Phi_s(N,j)=(a_1,\ldots,a_s)\) is precisely the exponent vector of the
\(j\)-th monomial in \(\mathcal{M}_{N,s}\).

Finally, for any data point \(\boldsymbol{x}_i\), the value of this monomial is
\[
\prod_{k=1}^s x_{ik}^{a_k}.
\]
By the definition of the univariate Vandermonde matrices,
\[
(V_{x_k})_{i,m_k}=x_{ik}^{m_k-1}=x_{ik}^{a_k},
\]
where \(m_k=a_k+1\). Therefore
\[
(\Vd)_{i,j}=\prod_{k=1}^s (V_{x_k})_{i,m_k(j)},
\]
which proves \eqref{eq:general-factor}.
\end{proof}

The universal quantifier in \(N\) is necessary: after the last exponent is
fixed, the induction call has degree \(N-a_s\), not generally \(N\).  The
bivariate result in Theorem~\ref{thm:two} is the first nontrivial instance of
this construction; its closed form gives a convenient base formula in two
dimensions.

\begin{algorithm}[t]
\caption{Direct recovery of the total-degree multi-index.}\label{alg:index}
\begin{algorithmic}[1]
\REQUIRE \(d,N,j\), where \(1\leq j\leq\binomc{N+d}{d}\)
\FOR{\(k=d,d-1,\ldots,1\)}
  \STATE Find the unique \(a_k\in\{0,\ldots,N\}\) such that
  \(f(N,k,a_k)<j\leq f(N,k,a_k+1)\).
  \STATE \(j\leftarrow j-f(N,k,a_k)\), \(N\leftarrow N-a_k\).
\ENDFOR
\RETURN \((a_1,\ldots,a_d)\), or \((a_1+1,\ldots,a_d+1)\) as univariate column indices.
\end{algorithmic}
\end{algorithm}

\begin{example}[\(d=2\), \(n=2\), \(j=4\)]
The total number of monomials is \(\binom{4}{2}=6\). The monomials in the
prescribed order are listed in Example~\ref{ex:order-2d}. We compute
\(\Phi_2(2,4)\) by applying Proposition~\ref{prop:block} with \(N=2\), \(s=2\),
and \(j=4\).

\begin{itemize}
    \item \(s=2\): \(N=2\), \(j=4\). Compute
    \[
    f(2,2,q)=\binom{4}{2}-\binom{4-q}{2}.
    \]
    Then \(f(2,2,0)=0\), \(f(2,2,1)=3\), and \(f(2,2,2)=5\). Since
    \(f(2,2,1)=3<4\le f(2,2,2)=5\), we get \(a_2=1\).
    Then \(j\leftarrow 4-3=1\), \(N\leftarrow 2-1=1\).
    \item \(s=1\): \(N=1\), \(j=1\). By the base case
    \(\Phi_1(N,j)=(j-1)\), we get \(a_1=0\).
\end{itemize}

Hence \(\Phi_2(2,4)=(a_1,a_2)=(0,1)\), corresponding to
\(x_1^0x_2^1=x_2\), which is the fourth monomial in the list of
Example~\ref{ex:order-2d}. The univariate column indices are
\((m_1,m_2)=(a_1+1,a_2+1)=(1,2)\).
\end{example}

\begin{example}[\(d=3\), \(n=2\), \(j=8\)]
The total number of monomials is \(\binom{5}{3}=10\). The monomials in the
prescribed order are listed in Example~\ref{ex:order-3d}. We compute
\(\Phi_3(2,8)\) by applying Proposition~\ref{prop:block} with \(N=2\), \(s=3\),
and \(j=8\).

\begin{itemize}
    \item \(s=3\): \(N=2\), \(j=8\). Compute
    \[
    f(2,3,q)=\binom{5}{3}-\binom{5-q}{3}.
    \]
    Then \(f(2,3,0)=0\), \(f(2,3,1)=6\), and \(f(2,3,2)=9\). Since
    \(f(2,3,1)=6<8\le f(2,3,2)=9\), we get \(a_3=1\).
    Then \(j\leftarrow 8-6=2\), \(N\leftarrow 2-1=1\).
    \item \(s=2\): \(N=1\), \(j=2\). Compute
    \[
    f(1,2,q)=\binom{3}{2}-\binom{3-q}{2}.
    \]
    Then \(f(1,2,0)=0\) and \(f(1,2,1)=2\). Since
    \(f(1,2,0)=0<2\le f(1,2,1)=2\), we get \(a_2=0\).
    Then \(j\leftarrow 2-0=2\), \(N\leftarrow 1-0=1\).
    \item \(s=1\): \(N=1\), \(j=2\). By the base case
    \(\Phi_1(N,j)=(j-1)\), we get \(a_1=1\).
\end{itemize}

Hence \(\Phi_3(2,8)=(a_1,a_2,a_3)=(1,0,1)\), corresponding to
\(x_1^1x_2^0x_3^1=x_1x_3\), which is the eighth monomial in the list of
Example~\ref{ex:order-3d}. The univariate column indices are
\((m_1,m_2,m_3)=(a_1+1,a_2+1,a_3+1)=(2,1,2)\).
\end{example}

\begin{example}[\(d=4\), \(n=3\), \(j=25\)]
The total number of monomials is \(\binom{3+4}{4}=35\). Ordered by the inverse
lexicographic order on \((a_4,a_3,a_2,a_1)\), they are:
\[
\begin{aligned}
&1,\ x_1,\ x_1^2,\ x_1^3,\ x_2,\ x_1x_2,\ x_1^2x_2,\ x_2^2,\ x_1x_2^2,\ x_2^3,\\
&x_3,\ x_1x_3,\ x_1^2x_3,\ x_2x_3,\ x_1x_2x_3,\ x_2^2x_3,\ x_3^2,\ x_1x_3^2,\ x_2x_3^2,\ x_3^3,\\
&x_4,\ x_1x_4,\ x_1^2x_4,\ x_2x_4,\ x_1x_2x_4,\ x_2^2x_4,\ x_3x_4,\ x_1x_3x_4,\ x_2x_3x_4,\ x_3^2x_4,\\
&x_4^2,\ x_1x_4^2,\ x_2x_4^2,\ x_3x_4^2,\ x_4^3.
\end{aligned}
\]
The monomial in position \(j=25\) is \(x_1x_2x_4\). We recover its
multi-index with the recursive procedure.

\begin{itemize}
    \item \(s=4\): \(N=3\), \(j=25\). Compute
    \[
    f(3,4,q)=35-\binom{7-q}{4}.
    \]
    Then \(f(3,4,1)=20\) and \(f(3,4,2)=30\). Since
    \(f(3,4,1)=20<25\le f(3,4,2)=30\), we get \(a_4=1\).
    Then \(j\leftarrow 25-20=5\), \(N\leftarrow 3-1=2\).
    \item \(s=3\): \(N=2\), \(j=5\). Compute
    \[
    f(2,3,q)=10-\binom{5-q}{3}.
    \]
    Then \(f(2,3,0)=0\) and \(f(2,3,1)=6\). Since
    \(f(2,3,0)=0<5\le f(2,3,1)=6\), we get \(a_3=0\).
    Then \(j\leftarrow 5-0=5\), \(N\leftarrow 2-0=2\).
    \item \(s=2\): \(N=2\), \(j=5\). Compute
    \[
    f(2,2,q)=6-\binom{4-q}{2}.
    \]
    Then \(f(2,2,1)=3\) and \(f(2,2,2)=5\). Since
    \(f(2,2,1)=3<5\le f(2,2,2)=5\), we get \(a_2=1\).
    Then \(j\leftarrow 5-3=2\), \(N\leftarrow 2-1=1\).
    \item \(s=1\): \(N=1\), \(j=2\). By the base case
    \(\Phi_1(N,j)=(j-1)\), we get \(a_1=1\).
\end{itemize}

Hence \(\Phi_4(3,25)=(a_1,a_2,a_3,a_4)=(1,1,0,1)\), corresponding to
\(x_1^1x_2^1x_3^0x_4^1=x_1x_2x_4\), as anticipated. The univariate column
indices are \((m_1,m_2,m_3,m_4)=(a_1+1,a_2+1,a_3+1,a_4+1)=(2,2,1,2)\).
\end{example}

\section{Cost and implementation consequences}

At each level \(k\), the value \(a_k\) is the integer solution of
\(\binom{M}{k}<K\le\binom{M+1}{k}\) with \(M=N+k-a_k-1\) and
\(K=\binom{N+k}{k}-j+1\). For \(k=1\) this inversion is trivial, and for
\(k=2\) it reduces to a quadratic inequality, which yields the closed form
used in the bivariate case. For \(k\ge 3\), however, inverting the binomial
coefficient amounts to solving a polynomial equation of degree \(k\). No
general closed-form solution by radicals exists for \(k\ge 5\), so binary
search on the strictly increasing function \(f(N,k,\cdot)\) is the practical
general-purpose method. This is standard in combinatorial number systems. Then, one probes the midpoint of the current interval and discards the half that cannot contain the transition. The number of evaluations of \(f\) is therefore \(\mathcal{O}(\log N)\) per level, and the total cost of recovering one multi-index is \(\mathcal{O}(d\log n)\) evaluations of \(f\), with \(\mathcal{O}(d)\) storage for the output vector.

Our multi-index recovery procedure is especially useful in three computational scenarios.

First, for \emph{random column access}: the ability to retrieve an arbitrary
column \(j\) of the Vandermonde matrix without computing or storing the
preceding columns. With a tensor-product construction, accessing column \(j\)
requires building the univariate matrices and then identifying which of the
\((n+1)^d\) candidate columns corresponds to the \(j\)-th monomial of the
total-degree space. With the present algorithm, the multi-index
\((a_1,\ldots,a_d)\) is recovered directly from \(j\), so any column can be
evaluated on demand.

Second, for \emph{streaming construction}: the ability to process data points
one at a time, as they become available, without assembling the full matrix.
This is relevant when the data arrive sequentially, for instance from a sensor
or a disk read, and cannot all be held in memory. Since the multi-index
\((a_1,\ldots,a_d)\) depends only on the column index \(j\) and not on the
data, the row corresponding to each new point can be evaluated and consumed
immediately.

Third, for \emph{matrix-free evaluation}: the ability to compute the action of
the Vandermonde matrix on a vector without materializing the matrix itself.
This is essential in iterative methods such as Krylov subspace solvers
\cite{Saad2003} or least-squares algorithms, where only products
\(V\mathbf{c}\) and \(V^T\mathbf{z}\) are needed. The present algorithm
provides these products directly: for each column \(j\), the multi-index
\((a_1,\ldots,a_d)\) is recovered, the monomial is evaluated at the relevant
points, and the contribution is accumulated. No dense matrix is ever formed.

In all three scenarios, the common bottleneck is not the evaluation of a
monomial but the identification of the monomial associated with a given
column. The present algorithm resolves this identification directly, without
constructing the full tensor-product candidate array. It does not store a
table of \(\binom{n+d}{d}\) multi-indices, nor does it materialize the full
tensor-product candidate matrix. For comparison, a tensor product of \(d\)
degree-\(n\) univariate matrices has \((n+1)^d\) columns, whereas the required
total-degree matrix has only \(\binom{n+d}{d}\) columns. For instance, with
\(n=3\) and \(d=4\), the tensor product has \(4^4=256\) columns while the
total-degree space has \(\binom{7}{4}=35\).


We conclude this section with an explicit comparison of three procedures for
constructing entries of the multivariate Vandermonde matrix: the multi-index
recovery procedure introduced in Algorithm~\ref{alg:index}, the Kronecker
product construction, and nested loops over the exponent multi-indices. The
last of these is a standard technique for generating multi-indices in
total-degree order; see \cite{deBoor2000} for a detailed treatment of the
indexing problem.

\medskip
\noindent\textbf{Single element \((i,j)\).}

\begin{itemize}
    \item \emph{Multi-index recovery.} Recover \((a_1,\ldots,a_d)\) from \(j\)
    by \(d\) binary searches on \(f\), then evaluate
    \(\prod_{k=1}^d x_{ik}^{a_k}\). The index recovery costs
    \(\mathcal{O}(d\log n)\) evaluations of \(f\). Evaluating the monomial
    costs \(\mathcal{O}(d\log n)\) with binary exponentiation, or
    \(\mathcal{O}(d)\) if the powers \(x_{ik}^m\) are already available.
    Total: \(\mathcal{O}(d\log n)\) per element in the worst case, or
    \(\mathcal{O}(d)\) with precomputed powers.

    \item \emph{Kronecker product.} Suppose $r$ data points. If the univariate matrices
    \[V_{x_1},\ldots,V_{x_d}\] are not yet built, constructing them costs
    \(\mathcal{O}(r\,d\,n)\) which is disproportionate for a single element.
    If they are already available, the entry is the product of \(d\)
    precomputed factors, costing \(\mathcal{O}(d)\) multiplications, but the
    mapping \(j\mapsto(a_1,\ldots,a_d)\) is still required, i.e., the same
    indexing problem that the multi-index recovery procedure solves. Hence,
    for an isolated element, Kronecker offers no advantage: without
    precomputed matrices it is far more expensive, and with them the indexing
    step remains.

        \item \emph{Nested loops.} Nested loops enumerate the monomials in the
    prescribed order by iterating over the exponent of the last variable, then
    the previous one, and so on, recursing on the number of variables.To reach the \(j\)-th monomial, this enumeration must visit all preceding
monomials, so it performs \(\mathcal{O}(j)\) steps. In the worst case,
\(j=\binom{n+d}{d}\), so the cost is
\(\mathcal{O}\bigl(\binom{n+d}{d}\bigr)\). Evaluating the monomial
    then adds \(\mathcal{O}(d\log n)\) with binary exponentiation, or
    \(\mathcal{O}(d)\) with precomputed powers.

    Thus nested loops are competitive with multi-index recovery only when
    \(j\) is small. The reason is that the two methods have different cost
    profiles: nested loops cost \(\mathcal{O}(j)\), proportional to the
    position of the target column, whereas multi-index recovery costs
    \(\mathcal{O}(d\log n)\), independent of \(j\). For small \(j\), counting
    up to \(j\) may be cheaper than performing \(d\) binary searches; for
    large \(j\), the enumeration dominates and the multi-index recovery is
    vastly superior.

    For example, take \(n=3\) and \(d=4\), so that the total number of
    monomials is \(\binom{7}{4}=35\). For \(j=1\), nested loops make a single
    step, while multi-index recovery performs about \(d\log_2 n \approx 4
    \cdot 2 = 8\) comparisons; here nested loops are faster. For
    \(j=35\), the last monomial, nested loops make 35 steps while multi-index
    recovery still performs about 8 comparisons; here multi-index recovery is
    much faster. This is why nested loops are the method of choice when all
    monomials are to be visited in order, but not for random access.
\end{itemize}

\medskip
\noindent\textbf{Full matrix (total-degree columns only).}

\begin{itemize}
    \item \emph{Multi-index recovery.} Precompute the powers \(x_{ik}^m\) for
    \(k=1,\ldots,d\), \(m=0,\ldots,n\), \(i=1,\ldots,r\); this costs
    \(\mathcal{O}(r\,d\,n)\) multiplications. Then, for each of the
    \(\binom{n+d}{d}\) columns, recover \((a_1,\ldots,a_d)\) in
    \(\mathcal{O}(d\log n)\) and evaluate the monomial at all \(r\) points in
    \(\mathcal{O}(r\,d)\) multiplications. The total cost is
    \[
    \mathcal{O}\!\left(r\,d\,\binom{n+d}{d}
    + d\log n\,\binom{n+d}{d}\right),
    \]
    where the precomputation cost \(\mathcal{O}(r\,d\,n)\) has been absorbed
    into the first term, since \(\binom{n+d}{d}\ge n+1\). Since $r$ is tipicaly much larger than $\log n$, the first term dominates, and the per-column index recovery adds a
    lower-order overhead.

    \item \emph{Kronecker product.} Build the \(d\) univariate matrices, which
    costs \(\mathcal{O}(r\,d\,n)\).  Let us use the row-wise Kronecker product (also known as the Khatri--Rao product), which is the
standard tool for constructing multivariate Vandermonde matrices in polynomial
interpolation and regression; see, e.g., \cite{Safak2012, Drmac2020} and the
discussion in \cite{KULeuven}. In particular, the multivariate
Vandermonde matrix is obtained by selecting the rows of the Khatri--Rao product
of the univariate Vandermonde matrices that correspond to monomials of total
degree at most \(n\). To start counting, we first form the full row-wise tensor product of
    \(d\) matrices: the intermediate product of the first \(k\) matrices has
    \(r\,(n+1)^k\) entries, and each step multiplies by one new factor. The
    total cost is
    \[
    \sum_{k=1}^{d} r\,(n+1)^k = \mathcal{O}\!\left(r\,(n+1)^d\right).
    \]
    The full product has \((n+1)^d\) columns, of which only
    \(\binom{n+d}{d}\) belong to the total-degree space. Selecting them costs
    an additional \(\mathcal{O}\!\left(r\,(n+1)^d\right)\) comparisons or
    copies. Total: \(\mathcal{O}\!\left(r\,(n+1)^d\right)\). For large \(d\), this is asymptotically worse than the multi-index recovery,
because \((n+1)^d\) grows exponentially in \(d\), whereas
\(d\binom{n+d}{d}\) grows only polynomially in \(d\) (of degree \(n+1\) for
fixed \(n\)).

        \item \emph{Nested loops.} Enumerate the monomials in the prescribed order
    using \(d\) nested loops (or recursion on the number of variables). Since
    consecutive monomials differ in a single exponent, the value of the current
    monomial at each of the \(r\) points can be updated incrementally: when the
    exponent of \(x_k\) increases by one, multiply the current value by
    \(x_{ik}\); when it decreases, divide by \(x_{ik}\). Each update costs
    \(\mathcal{O}(r)\) operations, and there are \(\binom{n+d}{d}\) monomials,
    giving a total of
    \(\mathcal{O}\!\left(r\,\binom{n+d}{d}\right)\). This avoids both the
    per-column index-recovery overhead of multi-index recovery and the storage
    of precomputed powers. Nested loops are therefore the method of choice when
    the full matrix is required and no random access is needed.
\end{itemize}




\begin{table}[t]
\centering
\caption{Cost comparison of the three strategies, in the best case for each
mode of access. Here \(r\) is the number of data points, \(d\) the number of
variables, \(n\) the total degree bound, and \(j\) the column index.}
\label{tab:cost-comparison}
\begin{tabular}{l c c}
\toprule
Strategy & Single element \((i,j)\) & Full matrix \\
\midrule
Multi-index recovery & \(\mathcal{O}(d\log n)\) & \(\mathcal{O}\!\left(r\,d\,\binom{n+d}{d}\right)\) \\
Kronecker product & \(\mathcal{O}(r\,d\,n)\) & \(\mathcal{O}\!\left(r\,(n+1)^d\right)\) \\
Nested loops & \(\mathcal{O}(j)\) & \(\mathcal{O}\!\left(r\,\binom{n+d}{d}\right)\) \\
\bottomrule
\end{tabular}
\end{table}

In summary, nested loops are the method of choice for the full matrix, while
multi-index recovery is the method of choice for random access, streaming,
and matrix-free evaluation, where the monomial associated with a given column
\(j\) must be identified without enumerating the preceding columns.

\section{Well-posedness and data geometry}
\label{sec:wellposedness}

The multi-index recovery procedure assumes that the columns of the
Vandermonde matrix are linearly independent. If they are not, the
least-squares problem does not have a unique solution, and the conditioning
of the matrix deteriorates. In this section we describe the geometric
conditions under which this happens.

Suppose \(d=3\) and all data points lie on a plane in \(\mathbb{R}^3\). The
restriction of the polynomial space of total degree at most \(n\) to that
plane has dimension \(\binom{n+2}{2}\), not \(\binom{n+3}{3}\). Hence the
Vandermonde matrix, which has \(\binom{n+3}{3}\) columns, cannot have full
column rank: its rank is at most \(\binom{n+2}{2}\), and the least-squares
problem has infinitely many solutions.

More generally, let \(k\) be the smallest degree of a nonzero polynomial that
vanishes at every data point. If \(k>n\), then no nonzero polynomial of total
degree at most \(n\) vanishes at every data point, and the Vandermonde matrix
retains full column rank, provided enough points in general position are
available.

Suppose now that \(k\le n\), and let \(L\) be a polynomial of total degree
\(k\) that vanishes at every data point. For every polynomial \(q\) of total
degree at most \(n-k\), the product \(Lq\) is a nonzero polynomial of total
degree at most \(n\) that vanishes at every data point. Writing
\[
Lq = \sum_{|\boldsymbol{\alpha}|\le n} c_{\boldsymbol{\alpha}}\,
\boldsymbol{x}^{\boldsymbol{\alpha}},
\]
we obtain a linear relation
\[
\sum_{|\boldsymbol{\alpha}|\le n} c_{\boldsymbol{\alpha}}\,
V_{i,\boldsymbol{\alpha}} = 0
\qquad\text{for all }i,
\]
so the coefficient vector \(c\) lies in the kernel of the Vandermonde matrix
\(V\). The number of distinct monomials \(q\) of total degree at most \(n-k\)
in \(d\) variables is
\[
\binom{n-k+d}{d}.
\]
Since the polynomial ring is an integral domain, multiplication by the nonzero
polynomial \(L\) is injective: if \(Lq_1=Lq_2\), then \(L(q_1-q_2)=0\), which
forces \(q_1=q_2\). Hence the products \(Lq\) are distinct and linearly
independent. Therefore the kernel of \(V\) has dimension at least
\[
\binom{n-k+d}{d},
\]
and the rank of \(V\) is at most
\[
\binom{n+d}{d} - \binom{n-k+d}{d}.
\]

This bound is exact when the vanishing ideal of the data is generated by the
single polynomial \(L\), that is, when every polynomial of total degree at most
\(n\) that vanishes at every data point is a multiple of \(L\). In that case
the only such polynomials are
\[
Lq,\qquad \deg q\le n-k,
\]
and there are \(\binom{n-k+d}{d}\) linearly independent ones. Hence
\[
\dim\ker V = \binom{n-k+d}{d},
\]
and the rank of \(V\) is exactly
\[
\binom{n+d}{d}-\binom{n-k+d}{d}.
\]
If there are several independent polynomials of degree \(k\) that vanish at the
data, the multiples of a single \(L\) provide only a lower bound: the kernel may
be larger and the rank smaller. For example, if the data lie on a line in
\(\mathbb{R}^3\), then \(k=1\), but both \(x\) and \(y\) may vanish at the
data, giving a larger kernel than predicted by a single linear polynomial.

This issue is independent of the indexing method: it is a property of the
fitting problem itself. Our procedure remains valid for indexing purposes, but
the user should be aware of the geometric conditions required for a
well-posed fitting problem. In practice, one can either add data points
outside the degenerate set, reduce the degree of the polynomial, or use
regularization to obtain a stable solution.

\section{Basis independence and Chebyshev extension}

The indexing map depends only on the ordered multi-index set, not on the
polynomial basis used to represent the monomials. Let
\(p_0,p_1,\ldots,p_n\) be any family of univariate polynomials with
\(\deg p_s=s\), and define the univariate matrices
\[
(W_{x_k})_{i,t}=p_{t-1}(x_{ik}),\qquad 1\leq t\leq n+1.
\]
Thus \(W_{x_k}\) is the univariate Vandermonde matrix in the basis
\((p_0,\ldots,p_n)\), evaluated at the data points \(x_{ik}\). For a column
\(j\) of the multivariate matrix, let \(\boldsymbol{a}(j)=(a_1(j),\ldots,a_d(j))\)
be the multi-index recovered by Algorithm~\ref{alg:index}, and let
\(m_k(j)=a_k(j)+1\) be the corresponding univariate column indices. Then the
multivariate matrix \(W^{(d)}\) in the basis
\(\{p_{a_1}(x_1)\cdots p_{a_d}(x_d)\}\) satisfies
\[
 W^{(d)}_{i,j}=\prod_{k=1}^d(W_{x_k})_{i,m_k(j)}.
\]
In particular, choosing \(p_s=T_s\), the Chebyshev polynomial of the first
kind, produces the total-degree Chebyshev--Vandermonde matrix with no change to
Algorithm~\ref{alg:index}. This separation is useful because numerical
stability is a basis and sampling issue distinct from combinatorial indexing;
orthogonal polynomial bases are often preferred when conditioning is the main
concern~\cite{BrubeckNakatsukasaTrefethen2021,KuianReichelShiyanovskii2020}.

\section{Conclusions}

We have introduced a multi-index recovery procedure for multivariate
Vandermonde matrices in the total-degree inverse lexicographic ordering. Given a global column
index \(j\), our procedure directly yields the associated univariate column
indices \(m_k(j)\) without enumerating the preceding columns. In two variables
it produces the explicit factors \(m(j)\) and \(\ell(j)\); in arbitrary
dimension it produces the recursively defined map \(\Phi_d\). The proof
establishes the correct exponents block boundaries, the validity of every local index,
and the factorization through univariate evaluation matrices.

Our indexing map depends only on the ordered multi-index set, not on the
polynomial basis. Consequently, the same algorithm applies to the monomial
basis and to alternative bases such as Chebyshev polynomials, where only the
univariate matrices change. This separation is useful because numerical stability depends on the choice of basis and sampling points, whereas the indexing map depends only on the
combinatorial structure of the multi-index set. The two concerns can therefore
be addressed independently.

We have also compared our multi-index recovery procedure with the Kronecker
product construction and with nested loops. For a single element, our
procedure costs \(\mathcal{O}(d\log n)\), whereas Kronecker requires building
the univariate matrices and nested loops require enumerating the preceding
columns. For the full matrix, nested loops with incremental updates achieve
the best asymptotic cost, \(\mathcal{O}(r\binom{n+d}{d})\), while Kronecker
pays the factor \((n+1)^d\) instead of \(d\binom{n+d}{d}\) with the procedure inhere. The advantage of our
procedure is therefore not in the full-matrix construction but in random
access, streaming, and matrix-free evaluation, where the monomial associated
with a given column must be identified without enumerating the preceding
columns.

The main application of our procedure is multivariate polynomial fitting, where
a design matrix with \(\binom{n+d}{d}\) columns must be constructed from
scattered or gridded data. This problem arises in several branches of physics
and engineering. A representative example is the construction of equations of
state for nuclear matter in neutron-star simulations. In that setting, the
thermodynamic quantities are functions of two or more variables, such as
baryon density, proton fraction, and temperature, and they are often
represented by polynomial fits to tabulated data~\cite{Servignat2025}. The
ComPOSE database, for instance, provides such tables, and converting them into
a smooth polynomial form suitable for numerical evolution requires exactly
the kind of multivariate Vandermonde matrix that our procedure
indexes~\cite{ComPOSE}. Similar needs arise in the approximation of
pairwise forces and torques between anisotropic particles, where multivariate
polynomials in the relative position and orientation coordinates are
used~\cite{Fakhraei2026}. In these and other applications, the ability to
recover the multi-index of a given column directly is essential for building
the design matrix efficiently, especially when the number of variables is
large and the tensor-product approach would produce a prohibitive number of
superfluous columns.

A numerical study comparing memory use, wall-clock time, and conditioning for
monomial and orthogonal bases is a natural next step. Another direction is to
extend our multi-index recovery procedure to other polynomial spaces, such as
weighted total-degree spaces (where the degree bound involves weights for each
variable) and simplex-like domains (where the monomial set is not a tensor
product). In both cases, the block structure would change, and a different
combinatorial decomposition would be needed to define the indexing map.

\section*{Acknowledgments}

The author thanks Jérôme Novak for suggesting the problem of multivariate
polynomial fitting in the context of nuclear matter equations of state, and
Dionisio Yáñez for pointing out the connection with the Kronecker product. This work is supported by the Spanish Agencia Estatal de Investigaci\'on (grant PID2024-159689NB-C21) funded by the Ministerio de Ciencia, Innovaci\'on y Universidades. The author also acknowledges support by the Generalitat Valenciana through the Grant ACIF/2019/169.

\bibliographystyle{siamplain}
\bibliography{vandermonde_references}
\end{document}